\documentclass[sn-mathphys,Numbered]{sn-jnl}

\usepackage{graphicx}%
\usepackage{multirow}%
\usepackage{amsmath}%
\usepackage{amssymb}
\usepackage{mathtools}
\usepackage{amsthm}%
\usepackage{mathrsfs}%
\usepackage[title]{appendix}%
\usepackage{xcolor}%
\usepackage{textcomp}%
\usepackage{manyfoot}%
\usepackage{booktabs}%
\usepackage{algorithm}%
\usepackage{algorithmicx}%
\usepackage{algpseudocode}%
\usepackage{listings}%

\theoremstyle{thmstyleone}%
\newtheorem{theorem}{Theorem}
\theoremstyle{thmstyletwo}%
\newtheorem{example}{Example}%
\newtheorem{remark}{Remark}%
\newtheorem{lemma}{Lemma}%
\theoremstyle{thmstylethree}%

\begin{document}

\title[Article Title]{A Modified Maximum Principle for Control Systems Driven by Mixed
Fractional Brownian Motion}

\author[1]{\fnm{Yuecai} \sur{Han}}\email{hanyc@jlu.edu.cn}

\author*[1]{\fnm{Yuhang} \sur{Li}}\email{yuhangl22@mails.jlu.edu.cn}

\equalcont{These authors contributed equally to this work.}

\affil[1]{\orgdiv{School of Mathematics}, \orgname{Jilin University}, \orgaddress{ \city{Changchun}, \postcode{130012},  \state{Jilin Province}, \country{China}}}


\abstract{In this paper,  we investigate the optimal control problem for systems driven by mixed fractional Brownian motion (including a fractional Brownian motion with Hurst parameter $H>1/2$ and the  standard Brownian motion). By using Malliavin calculus and introducing a disturbance control region, we obtain a modified maximum principle. Through martingale representation theorem, we obtain the adjoint backward stochastic differential equation in a natural way.
 Furthermore,  corresponding to \cite{han2013maximum}, a significant result is that the necessary condition is simplified by only containing one equality. As an application, the linear quadratic case is investigated to illustrate the main results.}

\keywords{maximum principle, mixed fractional Brownian motion, backward stochastic differential equation, linear quadratic control system}


\pacs[MSC Classification]{60G22, 60H07, 93E20}

\maketitle

\section{Introduction}
Let $H\in(1/2,\,1)$ be a fixed constant. The $m$-dimensional fractional Brownian motion $B_t^H=\left(B_1^H(t),\cdots,B_m^H(t)\right), t\in[0,T]$  of Hurst parameter $H$ is a continuous, mean 0 Gaussian process with the covariance
$$ \mathbb{E}[B_i^H(t)B_j^H(s)]=\frac{1}{2}\delta_{ij}(t^{2H}+s^{2H}-| t-s|^{2H}), $$
where
\begin{equation}
\delta_{ij}=\left\{
\begin{aligned}
&1,\ \ {\rm if}\ i=j,
\\
&0,\ \ {\rm if}\ i\neq j, 
\end{aligned}
\right.
\notag
\end{equation}
$i,\,j=1,\dots,m$. 

The stochastic differential equations (SDEs in short) driven by fractional Brownian motion refer to \cite{lyons2002system,coutin2002stochastic,rascanu2002differential,hu2007differential,hu2009rough,friz2010multidimensional,hu2013multiple} and references therein. The classical stochastic optimal control problem for system driven by standard Brownian motion has been popularly investigated, some initial works refer to \cite{kushner1972necessary,bismut1978introductory,peng1990general,yong1999stochastic}. For the optimal control of linear systems involving
fractional Brownian motion, Hu and Zhou \cite{hu2005stochastic}  obtain the explicit Markov control by using the technique of completing squares and  Riccati equations. Hu and Peng \cite{hu2009backward} investigate backward stochastic differential equations (BSDEs in short) driven by fractional Brownian motions. Duncan and Pasik-Duncan \cite{duncan2013linear} get an explicit feedback optimal control for a linear system driven by fractional Gaussian process and a cost functional that is quadratic on the
system state and the control.  Han et al  \cite{han2013maximum} obtain the maximum principle for general control system driven by fractional Brownian motion. The adjoint equation they got is a type of linear BSDEs driven by both fractional Brownian motion and  standard Brownian motion.  Sun \cite{sun2021stochastic} obtains a stochastic maximum principle for general control systems driven by fractional Brownian motions, in which the admissible control domain could be non-convex.

Let a fractional Brownian motion $B_t^H=\left(B_1^H(t),\cdots,B_m^H(t)\right), t\in[0,T]$ and a standard Brownian motion $B_t=\left(B_1(t),\cdots,B_m(t)\right), t\in[0,T]$ satisfy
\begin{align}
B_j^H(t)=\int_0^t Z_H(t,\,s) dB_j(s), \quad 1\le j\le m,
\end{align}
where 
\begin{equation*}
Z_H(t,\,s)=\kappa_H\left[\left(\frac{t}{s}\right)^{H-\frac{1}{2}}(t-s)^{H-\frac{1}{2}}-\left(H-\frac{1}{2}\right)s^{\frac{1}{2}-H}\int_s^t u^{H-\frac{3}{2}}(u-s)^{H-\frac{1}{2}}du\right]
\end{equation*}
with
\begin{equation*}
\kappa_H=\sqrt{\frac{2H\varGamma(\frac{3}{2}-H)}{\varGamma(H+\frac{1}{2})\varGamma(2-2H)}}.
\end{equation*}
Then we call $B_t$ is the underlying Brownian motion of $B_t^H$. The considered mixed fractional Brownian motion consists of a fractional Brownian motion and its underlying Brownian motion.

In this paper, we study the optimal control problem for a general control system driven by mixed fractional Brownian motion with the following state equation:
\begin{equation*}
    \left\{\begin{array}{l}
 d X(t)=b\left(t,X(t),u(t)\right)dt+\sigma\left(t,X(t),u(t)\right)dB(t)\\
\qquad\qquad+\gamma\left(t,X(t),u(t)\right)\circ dB^H(t),\quad t\in[0,T],\\
 X(0)=x_0.
\end{array}\right.
\end{equation*}
We introduce a modified admissible control domain and show it is density in classical control domain. So that, a necessary condition, which contains only one equality, for optimal control process is obtained.

A difficulty we encounter is that the method of fractional calculus 
used in \cite{han2013maximum} is hard to get a convergence result of the variation equation. To deal with it, we derive the  convergence result by means of constructing a stopping time. Compared with classical duality approach, we do not need to know the form of the adjoint equation in advance in our derivations.  We  give out the explicit solution to the variation equation by  introducing two SDEs with fractional It$\hat{\rm o}$'s formula. Then, the necessary condition the optimal control should satisfy is obtained by introducing a pair of adapted processes. We prove that this pair of processes satisfy a linear BSDE by martingale representation theorem and  Malliavin calculus. 
Indeed, we give a representation of the unique solution  for a class of linear BSDEs driven by both fractional Brownian motion and standard Brownian motion. 
By increasing the dimension of noises, we also obtain the maximum principle for  control systems driven by fractional
Brownian motion and an independent Brownian motion.
Finally, as an application, we obtain  the unique optimal control process for linear
quadratic (LQ in short) control systems driven by mixed fractional Brownian motion.

The rest of this paper is organized as follows. In section 2, we introduce some definitions and propositions about fractional Brownian motion and Malliavin calculus. In section 3, we obtain the modified stochastic maximum principle  and get the adjoint equation through Malliavin calculus. In section 4, the linear quadratic case is investigated to illustrate the main results.

\section{Preliminaries}
 Let $(\varOmega,\,\mathcal F,\,P)$ be a complete stochastic space with a filtration $\{\mathcal F_t\}_{t\geq 0}$. $\mathcal F_t$ is generated by an $m$-dimensional fractional Brownian motion  $B_t^H=(B_1^H(t),\cdots,B_m^H(t),\ t\in[0,\,T])$ with the Hurst parameter $H\in(1/2,\,1)$.

Let the operators  $\mathbf{\Gamma}_{H,T}^*$ and $\mathbb{B}_{H,T}^*$ (\cite{hu2005integral}) be defined by 
\begin{align}
\mathbf{\Gamma}_{H,T}^*f(t)=\left(H-\frac{1}{2}\right)\kappa_H t^{\frac{1}{2}-H}\int_t^T u^{H-\frac{1}{2}}(u-t)^{H-\frac{3}{2}}f(u)du, \quad 0\leq t\leq T
\label{2.3}
\end{align}
and
\begin{align*}
\mathbb{B}_{H,T}^*f(t)=-\frac{2H\kappa_1}{\kappa_H}t^{\frac{1}{2}-H}\frac{d}{dt}\int_t^T (u-t)^{\frac{1}{2}-H}u^{H-\frac{1}{2}}f(u)du,
\end{align*}
where $f$ is a deterministic function and
\begin{align*}
\kappa_1=\frac{1}{2H\Gamma(H-\frac{1}{2})\Gamma(\frac{3}{2}-H)}.
\end{align*}
Furthermore, we have (Definition 6.1 in \cite{hu2005integral})
\begin{align}
\int_0^T f(t)dB_j^H(t):=\int_0^T (\mathbf{\Gamma}_{H,T}^*f)(t)dB_j(t)
\label{2.4}
\end{align}
and
\begin{align*}
\int_0^T f(t)dB_j(t):=\int_0^T (\mathbb{B}_{H,T}^*f)(t)dB^H_j(t).
\end{align*}

It is obviously that the filtration generated by the fractional Brownian motion $\{B_t^H\}_{t\in[0,\,T]}$ coincides with the one generated by the underlying Brownian motion $\{B_t\}_{t\in[0,\,T]}$, because
\begin{align*}
B_j^H(s)=\int_0^s\left(\mathbf{\Gamma^*_{H,s}}\mathbf{1}\right)(r)dB_j(r)
\end{align*}
is $\mathcal{F}_t$-measurable for $\forall\, 0\le s \le t\le T$. In the same way, by using operator $\mathbb{B}^*_{H,T}$, we show that
the filtration generated by $\{B_t\}_{t\in[0,\,T]}$ coincides with the one generated by $\{B_t^H\}_{t\in[0,\,T]}$. So
\begin{align*}
\mathcal{F}_t=\sigma(B^H_s,\ 0\leq s\leq t)=\sigma(B_s,\ 0\leq s\leq t).
\end{align*}

Next, we introduce the Malliavin derivatives \cite{hu2005integral}. Let $\xi_1,\cdots,\xi_k,\cdots$ be an orthonormal basis of $L^2([0,\,T])$ such that $\xi_k$, $k=1,2,\cdots$ are smooth functions on $[0,\,T]$. 
We denote $\tilde{\xi}_{j,l}=\int_0^T\xi_j(t)dB_l(t)$, where $j=1,2,...,$ and $l=1,...,m$. Let $\mathcal{P}$ be the set of all polynomials of the standard Brownian motions $B$ over interval $[0,T]$. Namely, $P$ contains all elements of the form
\begin{align*}
F(\omega)=f(\tilde{\xi}_{j_1,l_1},...,\tilde{\xi}_{j_n,l_n}),
\end{align*}
where $f$ is a polynomial of $n$ variables. The Malliavin derivative $D_s^lF$ is defined by
\begin{equation*}
D_s^lF=\sum_{i=1}^{n}\frac{\partial f}{\partial x_i}\left(\tilde{\xi}_{j_1,l_1},...,\tilde{\xi}_{j_n,l_n}\right)\xi_{j_i}(s)I_{\{l_i=l\}},\quad \forall 0\le s\le T.
\end{equation*}

Define $\eta_j=\mathbb{B}_{H,T}\xi_j$ and denote by $\mathcal{P}^H$ the set of all polynomial functional $\tilde{\eta}_{j,l}=\int_0^T\eta_j(t)dB_l^H(t)$. The elements G in $\mathcal{P}^H$ have the following form
\begin{align*}
G(\omega)=g(\tilde{\eta}_{j_1,l_1},...,\tilde{\eta}_{j_n,l_n}),
\end{align*}
where $g$ is a polynomial of $n$ variables. We define its Malliavin derivative $D_s^{H,l}G$ by
\begin{equation*}
D_s^{H,l}G=\sum_{i=1}^{n}\frac{\partial g}{\partial x_i}\left(\tilde{\eta}_{j_1,l_1},...,\tilde{\eta}_{j_n,l_n}\right)\eta_{j_i}(s)I_{\{j_i=l\}},\quad \forall 0\le s\le T.
\end{equation*}

Another useful Malliavin derivative $\mathbb{D}_t^{H,j}G$ is defined as follows 
\begin{equation}
\mathbb{D}_t^{H,j} G=\int_0^T \phi(s,\,t)D_s^{H,j} Gds, \ \ 0\leq t\leq T,
\label{2.5}
\end{equation}
where 
\begin{align}
\phi(s,\,t)=H(2H-1)\lvert s-t\lvert^{2H-2}.
\label{2.6}
\end{align}

We introduce the following normed spaces that will be used in the sequel.
Let $1/2<H<1$, $1-H<\alpha<1/2$ and $d\in \mathbb{N}^*$. $W_0^{\alpha,\infty}(0,\,T;\,\mathbb{R}^d)$ denotes the space of measurable functions $f:[0,\,T]\rightarrow \mathbb{R}^d$ such that 
\begin{equation}
\|f\|_{\alpha,\infty}:=\mathop{\sup}\limits_{0\leq t\leq T}\|f\|_{\alpha,t}<\infty,
\notag
\end{equation}
where
\begin{equation}\label{2.7}
\|f\|_{\alpha,t}:=\lvert f(t)\lvert+\int_0^t \frac{\lvert f(t)-f(s)\lvert}{(t-s)^{\alpha+1}}ds.
\end{equation}

Let $W_T^{1-\alpha,\infty}$ denote the space of measurable functions $f:[0,\,T]\rightarrow \mathbb{R}^d$ such that  
\begin{align*}
\|f\|_{1-\alpha,\infty,T}<\infty,
\end{align*}
where 
\begin{align}\label{norm}
\|f\|_{1-\alpha, \infty, t}:=\sup _{0 \leq u<v\le t}\left(\frac{|f(v)-f(u)|}{(v-u)^{1-\alpha}}+\int_u^v \frac{|f(y)-f(u)|}{(y-u)^{2-\alpha}} \mathrm{d} y\right) .
\end{align}

\section{The Maximum Principle}
In this section, we consider the following control system driven by mixed fractional Brownian motion
\begin{equation}
    \left\{\begin{array}{l}
 d X(t)=b\left(t,X(t),u(t)\right)dt+\sigma\left(t,X(t),u(t)\right)dB(t)\\
\qquad\qquad+\gamma\left(t,X(t),u(t)\right)\circ dB^H(t),\quad t\in[0,T],\\
 X(0)=x_0,
\end{array}\right.
\end{equation}
where the integral w.r.t fractional Brownian motion is Stratonovich's type and the integral w.r.t Standard Brownian motion is it$\hat{\rm o}$'s type. 
The cost function is
\begin{align}
J(u(\cdot))=\mathbb{E}\left[\int_0^Tf\left(t,X(t),u(t)\right)dt+g\left(X(T)\right)\right].
\end{align}
Here $b$, $\sigma$, and $\gamma $ are measurable functions on $[0,T]\times \mathbb{R}^d\times \mathbb{R}^k$ with values in $\mathbb{R}^d$, $\mathbb{R}^{d\times m}$ and $ \mathbb{R}^{d\times m}$, respectively.
$f$ and $g $ are measurable differential functions on $[0,T]\times \mathbb{R}^d\times \mathbb{R}^k$ and $\mathbb{R}^d$, respectively, with values in $\mathbb{R}$.

We denote by $\mathbb{U}$ the set of admissible control process
\textbf{u}$=(u(t))_{0\le t\le T}$ taking values in a given convex set $\textbf{D}\subset \mathbb{R}^k$ satisfying $E\int_0^T |u(t)|^4 dt <+\infty$ and $\mathbb{E}[D_r^Hu(t)]^2+\mathbb{E}\left[\mathbb{D}_s^HD^H_ru(t)\right]^2<+\infty,\, \forall t,s,r\in[0,T]$.

Denote
\begin{align*}
& l_x\left(t, x, u\right)=\left(\frac{\partial f_i\left(t, x, u\right)}{\partial x_j}\right)_{1 \leq i, j \leq n}, \\
& l_u\left(t, x, u\right)=\left(\frac{\partial f_i\left(t, x, u\right)}{\partial u_j}\right)_{1 \leq i \leq n, 1 \leq j \leq k}
\end{align*}
for $l=b,\sigma,\gamma.$

Assume that the continuously differentiabl functions $b,\sigma,\gamma$ satisfy the following conditions.
\begin{enumerate}
	\item[(H1)]  There exists a constant $L>0$ such that
\begin{align*}
&|b_x(t,x,u)+|b_u(t,x,u)|+|\sigma_x(t,x,u)+|\sigma_u(t,x,u)|\\
&\quad+|\gamma_x(t,x,u)+|\gamma_u(t,x,u)|\le L,\\
&|b_x(t,x,u)-b_x(t,y,v)|+|b_u(t,x,u)-b_u(t,y,v)|\\
&\quad+|\sigma_x(t,x,u)-\sigma_x(t,y,v)|+|\sigma_u(t,x,u)-\sigma_u(t,y,v)|\\
&\quad+|\gamma_x(t,x,u)-\gamma_x(t,y,v)|+|\gamma_u(t,x,u)-\gamma_u(t,y,v)|\\
&\le L\left(|x-y|+|u-v|\right),\quad \forall t\in[0,T], \quad\forall x,y\in\mathbb{R}^d, \quad\forall u,v\in\mathbb{R}^k.
\end{align*}
	
\item[(H2)] $\gamma(t,x,u)$ is H\"{o}lder continuous in time uniformly in $(x,u)$. There exists some constants $1-H<\gamma<1$, and $L>0$, such that
\begin{align*}
&|\gamma(t,x,u)-\gamma(s,x,u)|+|\gamma_x(t,x,u)-\gamma_x(s,x,u)|\\
&\quad+|\gamma_u(t,x,u)-\gamma_u(s,x,u)|\le L|t-s|^\gamma,\quad
\forall t,s\in[0,T],\quad \forall x\in\mathbb{R}^d,\quad u\in\mathbb{R}^k.
\end{align*}
\end{enumerate}

To simplify the notation without losing of the generality, we  consider the case $d=k=1$. We assume that $u^{*}(\cdot)$ is the optimal control process, i.e.,
\begin{align*}
    J(u^*(\cdot))=\inf_{u(\cdot)\in \mathbb{U}}J(u(\cdot))
,\end{align*}
and $X^*(t)$ is the corresponding state process. We call $\left(X^*(t),u^*(t)\right)$ an optimal pair.

Define
$$\mathbb{V}=\left\{v(\cdot): v(\cdot)=u^1(\cdot)-u^2(\cdot), \quad u^1(\cdot),u^2(\cdot)\in\mathbb{U}\right\}$$
and
$$\Bar{\mathbb{V}}=\left\{v(\cdot): v(\cdot)\in\mathbb{V}\quad s.t. \quad\mathbb{D}_t^Hv(t)=0,\quad\forall t\in[0,T]\right\}.$$
\begin{remark}
It is easy to check $\mathbb{V}=\mathbb{U}$ iff $\mathbf{D}=\mathbb{R}^k$.
\end{remark}
Then we show that $\Bar{\mathbb{V}}$ is density in $\mathbb{V}$, which is a key result in this paper. Namely, we have the following lemma.
\begin{lemma}\label{convergence}
For any $v(\cdot)\in \mathbb{V}$, there exists at least one sequence $\{v_n(\cdot)\}_{n=1}^{+\infty}$ in $\bar{\mathbb{V}}$ such that $\lim_{n\to\infty}\mathbb{E}|v_n(t)-v(t)|^2=0,\,\forall\, t\in[0,T]$.
\end{lemma}
\begin{proof}
Fix a $t\in[0,T]$, by Clark type representation (Chapter 12 of \cite{hu2005integral}), there exists the unique adapted square integrable process $\xi(\cdot)$ such that
\begin{align*}
v(t)=\mathbb{E}v(t)+\int_0^t\xi(s)dB^H(s),
\end{align*}
where the integral is It$\hat{\rm o}$-Wick's
type. Let $v_n(\cdot)$ be defined as
\begin{align*}
v_n(t)=\mathbb{E}v(t)+\int_0^t\xi_n(s)dB^H(s),
\end{align*}
where
\begin{align*}
\xi_n(s)=\xi(s)\mathbf{1}_{[0,t-\frac{1}{n}]}(s)-\frac{\int_0^{t-\frac{1}{n}}(t-r)^{2H-2}\xi(r)dr}{\int_{t-\frac{1}{n}}^t(t-r)^{2H-2}dr}\mathbf{1}_{(t-\frac{1}{n},t]}(s).
\end{align*}
It is clear that $\xi_n(\cdot)$ is $\mathcal{F}_t$-adapted and we  have
\begin{align*}
\mathbb{D}^{H}_t[v_n(t)]=&H(2H-1)\int_0^T|t-s|^{2H-2}D_s^{H}v_n(t)ds\\
=&H(2H-1)\int_0^t|t-s|^{2H-2}\xi_n(s)ds\\
=&H(2H-1)\int_0^{t-\frac{1}{n}}|t-s|^{2H-2}\xi(s)ds\\
&-H(2H-1)\int_{t-\frac{1}{n}}^t|t-s|^{2H-2}\frac{\int_0^{t-\frac{1}{n}}(t-r)^{2H-2}\xi(r)dr}{\int_{t-\frac{1}{n}}^t(t-r)^{2H-2}dr}ds
\\
=&0.
\end{align*}
Then we show the convergence. Denote $\tilde{\xi}_n(s)=\xi_n(s)-\xi(s)$, notice that $\mathbb{E}|\tilde{\xi}_n(s)|^2=O(n^{4H-2})$ and $\mathbb{E}|\mathbb{D}^H_s\tilde{\xi}_n(r)|^2=O(n^{4H-2}),\, s,r\in[t-\frac{1}{n},t]$. Through Proposition 2.1 of \cite{yuecai2018solutions}, we have
\begin{align*}
\mathbb{E}|v_n(t)-v(t)|^2=&\mathbb{E}\left|\int_{t-\frac{1}{n}}^t\tilde{\xi}_n(s)dB^H(s)\right|^2\\
=&\mathbb{E}\int_{t-\frac{1}{n}}^t\int_{t-\frac{1}{n}}^t\phi(s,t)\tilde{\xi}_n(s)\tilde{\xi}_n(r)dsdr\\
&+\mathbb{E}\int_{t-\frac{1}{n}}^t\int_{t-\frac{1}{n}}^t\mathbb{D}^H_s\tilde{\xi}_n(r)\mathbb{D}^H_r\tilde{\xi}_n(s)dsdr\\
\le&\mathbb{E}|\tilde{\xi}_n|^2\int_{t-\frac{1}{n}}^t\int_{t-\frac{1}{n}}^tH(2H-1)|s-r|^{2H-2}dsdr\\
&+\mathbb{E}\int_{t-\frac{1}{n}}^t\int_{t-\frac{1}{n}}^t\left|\mathbb{D}^H_s\tilde{\xi}_n(r)\right|^2dsdr\\
=&O\left(n^{-(2-2H)}\right)+O\left(n^{-(4-4H)}\right),
\end{align*}
which tends to $0$ as $n\to +\infty$.
This complete the proof of Lemma \ref{convergence}.
\end{proof}

\begin{example}
Let $u(t)=|B^H(t)|^2$, then we have
\begin{align*}
u(t)=t^{2H}+2\int_0^tB^H(s)dB^H(s).
\end{align*}
So that, we choose
\begin{align*}
u_n(t)=t^{2H}+2\int_0^t\xi_n(s)dB^H(s),
\end{align*}
where
\begin{align*}
\xi_n(s)=B^H(s)\mathbf{1}_{[0,t-\frac{1}{n}]}(s)-\frac{\int_0^{t-\frac{1}{n}}(t-r)^{2H-2}B^H(r)dr}{\int_{t-\frac{1}{n}}^t(t-r)^{2H-2}dr}\mathbf{1}_{(t-\frac{1}{n},t]}(s).
\end{align*}
Then it is easy to check $\mathbb{D}^H_tu_n(t)=0$ and $\lim_{n\to\infty}\mathbb{E}|u_n(t)-u(t)|^2=0$.
\end{example}

For all $0<\varepsilon<1$ and any other admissible control $u(\cdot)$, let 
\begin{align*}
    u^\varepsilon(t)=(1-\varepsilon)u^*(t)+\varepsilon u(t)\triangleq
u^*(t)+\varepsilon v(t),
\end{align*}
and $X^\varepsilon(t)$ is the corresponding state process. 
Based on the above lemma, it is reasonable to assume $v(\cdot)\in\bar{\mathbb{V}}$ at first.

Denote
\begin{align*}
l^*(t)=l(t,X^*(t),u^*(t))
\end{align*}
and 
\begin{align*}
l^\varepsilon(t)=l(t,X^\varepsilon(t),u^\varepsilon(t))
\end{align*}
for $l=b,\sigma,\gamma,f,b_x,\sigma_x,\gamma_x,f_x,b_u,\sigma_u,\gamma_u,f_u$.
Define the variation equation
\begin{equation}\label{3.3}
\left\{\begin{array}{l}
dy(t)=\left[b^*_x(t)y(t)+b^*_u(t)v(t)\right]dt+\left[\sigma^*_x(t)y(t)+\sigma^*_u(t)v(t)\right]dB(t)
\\\qquad\qquad+\left[\gamma^*_x(t)y(t)+\gamma^*_u(t)v(t)\right]
\circ dB^H(t),\quad t\in[0,T],\\
 y(0)=0.
\end{array}\right.
\end{equation}
Then, we have the following convergence result.

\begin{lemma}
    \label{lem1}
Assume $1/2<H<1$ and $1-H<\alpha<1/2$. Let assumptions {\rm (H1)}, {\rm (H2)} hold and
\begin{align}
	\tilde{X}^\varepsilon(t)=\frac{X^{\varepsilon}(t)-X^*(t)}{\varepsilon}-y(t).
\end{align}
	Then
\begin{align*}
\mathop{\lim}\limits_{\varepsilon\rightarrow 0}\sup_{0\le t\le T}\mathbb{E}\|\tilde{X}^\varepsilon(\cdot)\|^{2}_{\alpha ,t}=0.
\end{align*}
	Here, $\|\cdot\|_{\alpha,t}$ is defined in (\ref{2.7}).
\end{lemma}
\begin{proof}
	According to the definition of $\tilde X^\varepsilon(t)$,we have
	\begin{equation}\label{133}
	\left\{
	\begin{aligned}
	&d\tilde X^\varepsilon(t)=[\frac{1}{\varepsilon}\big(b^{\varepsilon}(t)-b^*(t)\big)-b_x^*(t)y(t)-b_u^*(t)v(t)]dt
	\\
	&\qquad\quad +[\frac{1}{\varepsilon}(\sigma^{\varepsilon}(t)-\sigma^*(t))-\sigma_x^*(t)y(t)-\sigma_u^*(t)v(t)]dB(t)
	\\
	&\qquad\quad +[\frac{1}{\varepsilon}(\gamma^{\varepsilon}(t)-\gamma^*(t))-\gamma_x^*(t)y(t)-\gamma_u^*(t)v(t)]\circ dB^H(t), &t&\in[0,\,T],
	\\
	&\tilde{X}^\varepsilon(0)=0.
	\end{aligned}
	\right.
\end{equation}
	After a first order development, denoting $X^{\lambda}(t)=X^*(t)+\lambda(X^{\varepsilon}(t)-X^*(t))$ and $u^{\lambda,\varepsilon}(t)=u^*(t)+\lambda\varepsilon v(t)$ for notation convenience, we have
\begin{align*}
d\tilde{X}^\varepsilon(t)=&\int_0^1 b_x(t,\,X^{\lambda}(t),u^{\varepsilon}(t))\tilde{X}^\varepsilon(t)d\lambda dt
+\int_0^1 \big(b_x(t,X^{\lambda}(t),u^{\varepsilon}(t))-b_x^*(t)\big)y(t)d\lambda dt
\\
&+\int_0^1 \big(b_u(t,X^*(t),u^{\lambda,\varepsilon}(t))-b_u^*(t)\big)v(t)d\lambda dt
\\
&+\int_0^1 \sigma_x(t,\,X^{\lambda}(t),\,u^{\varepsilon}(t))\tilde{X}^\varepsilon(t)d\lambda dB(t)\\
&+\int_0^1 \big(\sigma_x(t,\,X^{\lambda}(t),\,u^{\varepsilon}(t))-\sigma_x^*(t)\big)y(t) d\lambda dB(t)
	\\
	&+\int_0^1 \big(\sigma_u(t,\,X^*(t),\,u^{\lambda,\varepsilon}(t))-\sigma_u^*(t)\big)v(t) d\lambda dB(t)
	\\
	&+\int_0^1 \gamma_x(t,\,X^{\lambda}(t),\,u^{\varepsilon}(t))\tilde{X}^\varepsilon(t)d\lambda \circ dB^H(t)\\
 &+\int_0^1 \big(\gamma_x(t,\,X^{\lambda}(t),\,u^{\varepsilon}(t))-\gamma_x^*(t)\big)y(t) d\lambda\circ dB^H(t)
	\\
	&+\int_0^1 \big(\gamma_u(t,\,X^*(t),\,u^{\lambda,\varepsilon}(t))-\gamma_u^*(t)\big)v(t) d\lambda \circ dB^H(t)
	\\
	=&F(t)dt+G(t)dB(t)+H(t)\circ dB^H(t).
	\end{align*}	
	The last two terms in $F(t), G(t), H(t)$ tend to 0 in $L^2(\varOmega \times [0,\,T])$ as $\varepsilon\rightarrow 0$. To see this, we rewrite the last two terms in $F(t)$ as
	\begin{align*}
    I(t)=&\int_0^1 \big(b_x(t,\,X^*(t)+\lambda(X^{\varepsilon}(t)-X^*(t)),\,u^{\varepsilon}(t))-b_x(t,\,X^*(t),\,u^{\varepsilon}(t))\big)y(t) d\lambda \\
    &+\int_0^1 \big(b_x(t,\,X^*(t),\,u^*(t)+\lambda\varepsilon v(t))-b_x(t,\,X^*(t),\,u^*(t))\big)y(t)d\lambda
    \\
    &+\int_0^1 \big(b_u(t,\,X^*(t),\,u^*(t)+\lambda\varepsilon v(t))-b_u(t,\,X^*(t),\,u^*(t))\big)v(t)d\lambda.
\end{align*}	
	 By using  the Lipschitz continuity and boundedness of the functions as well as Cauchy-Schwarz inequality, we have
	 \begin{align*}
	 \mathbb{E}\left[\int_0^T \lvert I(t)\lvert^2 dt\right]
	 \le&\ C \left\{\left(\int_0^T \int_0^1 \mathbb{E}\lvert\lambda\varepsilon(\tilde{X}^\varepsilon(t)+ y(t))\lvert^4 d\lambda dt\right)^{\frac{1}{2}}\left(\int_0^T \mathbb{E}\lvert y(t)\lvert^4 dt \right)^{\frac{1}{2}}\right.
	 \\
	 &\left.+\left(\int_0^T \int_0^1 \mathbb{E}\lvert\lambda\varepsilon v(t)\lvert^4 d\lambda dt\right)^{\frac{1}{2}}\left(\int_0^T \mathbb{E}\lvert y(t)\lvert^4dt\right)^{\frac{1}{2}}\right.
	 \\
	 &\left.+\left(\int_0^T \int_0^1 \mathbb{E}\lvert\lambda\varepsilon v(t)\lvert^4 d\lambda dt\right)^{\frac{1}{2}}\left(\int_0^T \mathbb{E}\lvert v(t)\lvert^4dt\right)^{\frac{1}{2}}\right\},
	 \end{align*}
	 which converges to 0 as $\varepsilon\rightarrow 0$. Similar technique can be used to the last two terms in $G(t), H(t)$.
	 	 
	 From the above result, we rewrite (\ref{133}) as 
	 \begin{align*}
	 \tilde{X}^\varepsilon(t)=&\int_0^t\int_0^1 b_x(s,\,X^{\lambda}(s),\,u^{\varepsilon}(s))\tilde{X}^\varepsilon(s) d\lambda ds
	 \\
	 &+\int_0^t\int_0^1 \sigma_x(s,\,X^{\lambda}(s),\,u^{\varepsilon}(s))\tilde{X}^\varepsilon(s)d\lambda dB(s)
	 \\
	  &+\int_0^t\int_0^1 \gamma_x(s,\,X^{\lambda}(s),\,u^{\varepsilon}(s))\tilde{X}^\varepsilon(s)d\lambda \circ dB^H(s)+C_{\varepsilon}(t),
	 \end{align*}
	 where $C_{\varepsilon}(t)\rightarrow 0$ as $\varepsilon\rightarrow 0$, a.s.. Inspired by \cite{da2018mixed}, We
define the following stopping time
\begin{align*}
T_R:=\inf \left\{t\ge 0:\|B^H\|_{1-\alpha,\infty,t}\ge R\right\}\wedge T,
\end{align*}
where $\|\cdot\|_{1-\alpha,\infty,t}$ is defined as (\ref{norm}).
Notice that the family $(T_R)_{R\in \mathbb{N}^*}$ satisfies the property
\begin{align*}
\mathbb{P}\left[\bigcup_{R\in \mathbb{N}^*}\{T_R=T\}\right]=1,
\end{align*}
because of $P\big[\bigcap\{T_R<T\}\big]\le P\left[\|B^H\|_{1-\alpha,\infty,T}=\infty\right]=0.$
Denote
\begin{align*}
\tilde{X}^\varepsilon_R(t):=\tilde{X}^\varepsilon(t\wedge T_R), \quad t\in [0,T].
\end{align*}
Under the assumption (H1), we derive 
\begin{align*}
\mathbb{E}\| \tilde{X}^\varepsilon_R(\cdot) \|^{2}_{\alpha,t}
\leq & C\mathbb{E}\Bigg[\bigg\|\int_0^{\cdot\wedge T_R} \tilde X^\varepsilon(s)ds\bigg\|^{2}_{\alpha ,t}
+\bigg\|\int_0^{\cdot\wedge T_R} \tilde X^\varepsilon(s)dB(s)\bigg\|^{2}_{\alpha ,t}\\
&\qquad+\bigg\|\int_0^{\cdot\wedge T_R} \tilde X^\varepsilon(s)\circ dB^H(s)\bigg\|^{2}_{\alpha ,t}+|\Tilde{C}_{\varepsilon}|\Bigg]\\
=&C(|\Tilde{C}_{\varepsilon}|+A_1+A_2+A_3),
\end{align*}
where $|\Tilde{C}_{\varepsilon}|=\sup_{t\in[0,T]}\mathbb{E}\|C_\varepsilon\|_{\alpha,t}^{2},$ which tends to $0$ as $\varepsilon\to 0$. $C$ denotes a generic positive constant depend on  parameters of the problem, which may vary from line to line.

For $A_1$, we have
\begin{align*}
\left\|\int_0^{\cdot\wedge T_R} \tilde X^\varepsilon(s)ds\right\|_{\alpha ,t}\le& \int_0^{t\wedge T_R} |\tilde X^\varepsilon(s)|ds+\int_0^t(t-s)^{-\alpha-1}\int_{s\wedge T_R}^{t\wedge T_R}|\tilde{X}^\varepsilon(u)|duds\\
\le&\int_0^t|\tilde X^\varepsilon_R(s)|ds+\int_0^t(t-s)^{-\alpha-1}\int_s^t|\tilde{X}^\varepsilon_R(u)|duds\\
\le&\int_0^t|\tilde X^\varepsilon_R(s)|ds+\frac{1}{\alpha}\int_0^t(t-s)^{-\alpha}|\tilde X^\varepsilon_R(s)|ds\\
\le &\left(T^\alpha+\frac{1}{\alpha}\right)\int_0^t(t-s)^{-\alpha}|\tilde X^\varepsilon_R(s)|ds.
\end{align*}
Using the $
\text {Hölder}
$'s inequality and the fact $\alpha<\frac{1}{2}$, we obtain
\begin{align*}
A_1\le C\left(\int_0^t\mathbb{E}|\tilde X^\varepsilon_R(s)|^{2}ds\right).
\end{align*}

For $A_2$, we derive
\begin{align*}
A_2\le& C\mathbb{E}\left(\int_0^{t\wedge T_R}|\tilde X^\varepsilon(s)|dB(s)\right)^{2}\\
&+C\mathbb{E}\left(\int_0^t(t-s)^{-\alpha-1}\left|\int_{s\wedge T_R}^{t\wedge T_R}\tilde{X}^\varepsilon(u)dB(u)\right|ds\right)^{2}\\
=&A_{21}+A_{22}.
\end{align*}
Using the It$\hat{\rm o}$ isometry and the $
\text {Hölder}
$'s inequality, we have
\begin{align*}
A_{21}\le C\mathbb{E}\left(\int_0^{t\wedge T_R}|\tilde X^\varepsilon_R(s)|^{2}ds\right)\le C\left(\int_0^t\mathbb{E}|\tilde X^\varepsilon_R(s)|^{2}ds\right)
\end{align*}
and
\begin{align*}
A_{22}&\le C\left(\int_0^t\frac{ds}{(t-s)^{\alpha+\frac{1}{2}}}\right)\int_0^t(t-s)^{-\alpha-\frac{1}{2}-N}\mathbb{E}\left[\left|\int_{s\wedge T_R}^{t\wedge T_R}\tilde X^\varepsilon(u)dB(u)\right|^{2}\right]ds\\
&\le C\int_0^t(t-s)^{-\alpha-\frac{3}{2}}\mathbb{E}\left[\int_{s\wedge T_R}^{t\wedge T_R}|\tilde X^\varepsilon(u)|^{2}du\right]ds\\
&\le C\int_0^t(t-s)^{-\alpha-\frac{3}{2}}\mathbb{E}\left[\int_{s}^{t}|\tilde X^\varepsilon_R(u)|^{2}du\right]ds\\
&\le C\int_0^t(t-s)^{-\alpha-\frac{1}{2}}\mathbb{E}|\tilde X^\varepsilon_R(s)|^{2}ds.
\end{align*}

From proposition 4.3 of \cite{rascanu2002differential} and by using the $
\text {Hölder}
$'s inequality, we have
\begin{align*}
A_3\le CR^{2}\int_0^t\left((t-s)^{-2\alpha}+s^{-\alpha}\right)\mathbb{E}\|\tilde X^\varepsilon_R(\cdot)\|_{\alpha,s}^{2}ds.
\end{align*}
Above all, we obtain
\begin{align*}
\mathbb{E}\|\tilde X^\varepsilon_R(\cdot)\|_{\alpha,t}^{2}\le C|\tilde{C}_{\varepsilon}|+C\left(1+R^{2}\right)\int_0^t\varphi(t,s)\mathbb{E}\|\tilde X^\varepsilon_R(\cdot)\|_{\alpha,s}^{2}ds,
\end{align*}
where
\begin{align*}
\varphi(t,s)=s^{-\alpha}+(t-s)^{-\alpha-\frac{1}{2}}.
\end{align*}

Notice that $\varphi(t,s)\le (\sqrt{T}+1)t^{\alpha+\frac{1}{2}}(t-s)^{-\alpha-\frac{1}{2}}s^{-\alpha-\frac{1}{2}}$ for $0\le s\le t\le T$, We take the maximum of both sides of the above equation, so that
\begin{align*}
&\sup _{s \in[0, t]} \mathbb{E}\left[\left\|\tilde{X}^\varepsilon_R(\cdot)\right\|_{\alpha, s}^{2 }\right]\\
&\le C\left[|\tilde{C}_{\varepsilon}|+\left(1+R^{2}\right)(\sqrt{T}+1)t^{\alpha+\frac{1}{2}}\int_0^t(t-s)^{-\alpha-\frac{1}{2}}s^{-\alpha-\frac{1}{2}}\sup_{u\in[0,s]}\mathbb{E}\|\tilde X^\varepsilon_R(\cdot)\|_{\alpha,u}^{2}ds\right].
\end{align*}
By the Lemma 7.6 of \cite{rascanu2002differential}, we deduce
\begin{align*}
\sup _{s \in[0, t]} \mathbb{E}\left[\left\|\tilde{X}^\varepsilon_R(\cdot)\right\|_{\alpha, s}^{2 }\right] \leq C\left|\tilde{C}_{\varepsilon}\right| d_\alpha \exp \left(c_\alpha t\left[C\left(1+R^{2 }\right)(\sqrt{T}+1)\right]^{\frac{2}{1-2 \alpha}}\right)
\end{align*}
for all $t\in[0,T]$, where $c_\alpha,d_\alpha$ are positive constants depending only on $\alpha$. 
Let
\begin{align*}
R=R(\varepsilon):=\left(\frac{\left[-\frac{1}{2}ln\left(|\tilde{C}_{\varepsilon}|\wedge1\right)\right]^{\frac{1-2\alpha}{2}}}{C(\sqrt{T}+1)(c_\alpha T)^{\frac{1-2\alpha}{2}}}-1\right)^{\frac{1}{2}},
\end{align*}
which tends to $+\infty$ as $\varepsilon\to 0$. So that
\begin{align*}
\mathbb{E}\left[\left\|\tilde{X}^\varepsilon_R(\cdot)\right\|_{\alpha, t}^{2 }\right] \leq C\left|\tilde{C}_{\varepsilon}\right|^{\frac{1}{2}},\quad \forall t\in[0,T].
\end{align*}
Let $\varepsilon\to 0$, we complete the proof of Lemma \ref{lem1}.

\end{proof}

To derive the maximum principle, we need to obtain an explicit solution to Eq. (\ref{3.3}). Consider the following linear SDE.
\begin{equation}\label{3.6}
\left\{\begin{array}{l}
d\Phi(t)=b_x^*(t)\Phi(t)dt+\sum_{j=1}^m\left[\sigma_x^{j,*}(t)\Phi(t)dB_j(t)+\gamma_x^{j,*}(t)\Phi(t)\circ dB_j^H(t)\right],\\
 \Phi(0)=1.
\end{array}\right.
\end{equation}
It is easy to verify that $\Psi(t)=\Phi^{-1}(t)$ is the unique solution of the following SDE:
\begin{equation}\label{3.7}
\left\{\begin{array}{l}
d\Psi(t)=\left(-b_x^*(t)+\sum_{j=1}^m\sigma_x^{j,*}(t)^2\right)\Psi(t)dt-\sum_{j=1}^m\sigma_x^{j,*}(t)\Psi(t)dB_j(t)\\
\qquad\qquad-\sum_{j=1}^m\gamma_x^{j,*}(t)\Psi(t)\circ dB_j^H(t),\\
 \Psi(0)=1.
\end{array}\right.
\end{equation}

\begin{remark}
To calculate $\Psi$, we need to show $\Phi(t)\neq 0,\,\forall t\in[0,T]$. Indeed, by it$\hat{o}$'s formula, we have
\begin{align*}
dln\Phi(t)=&\Phi(t)^{-1}d\Phi(t)-\frac{1}{2}\Phi(t)^{-2}\sum_{j=1}^m\left(\sigma_x^{j,*}(t)\Phi(t)\right)^2dt\\
=&\left(b_x^*(t)+\frac{1}{2}\sum_{j=1}^m\sigma_x^{j,*}(t)^2\right)dt+\sum_{j=1}^m\left[\sigma_x^{j,*}(t)dB_j(t)+\gamma_x^{j,*}(t)\circ dB_j^H(t)\right].
\end{align*}
So that,
\begin{align*}
\Phi(t)=&exp\Bigg\{\int_0^tb_x^*(s)+\frac{1}{2}\sum_{j=1}^m\sigma_x^{j,*}(s)^2ds\\
&\qquad+\sum_{j=1}^m\int_0^t\sigma_x^{j,*}(s)dB_j(s)+\sum_{j=1}^m\int_0^t\gamma_x^{j,*}(s)\circ dB_j^H(s)\Bigg\}\\
>&0.
\end{align*}
\end{remark}

From the It$\hat{\rm o}$'s formula and Proposition 2.7 of \cite{han2013maximum} we get the solution of Eq. (\ref{3.3}) by using $\Phi(t)$ and $\Psi(t)$.
~\\

\begin{lemma}
Let $\Phi(t)$ and $\Psi(t)$ be defined by (\ref{3.6}) and (\ref{3.7}). Then the solution to Eq. (\ref{3.3}) is given by
\begin{align}\label{3.9}
y(t)=&\Phi(t)\int_0^t\Psi(s)\left(b_u^*(s)-\sum_{j=1}^m\sigma_x^{j,*}(s)\sigma_u^{j,*}(s)\right)v(s)ds\notag\\
&+\sum_{j=1}^m\Phi(t)\int_0^t\Psi(s)\sigma_u^{j,*}(s)v(s)dB_j(s)\notag\\
&+\sum_{j=1}^m\Phi(t)\int_0^t\Psi(s)\gamma_u^{j,*}(s)v(s)\circ dB_j^H(s).
\end{align}
\end{lemma}
Then we have the following theorem by using $\Phi, \Psi$ and $y$.
\begin{theorem}\label{mainin}
Let $u^*$ be the optimal admissible control. Then the following inequality holds
\begin{align}\label{3.10}
\mathbb{E}\int_0^T\Bigg[&\int_t^Tf_x^*(s)\Phi(s)ds\Psi(t)\left((b_u^*(t)-\sum_{j=1}^m\sigma_x^{j,*}(t)\sigma_u^{j,*}(t)\right)
+f_u^*(t)\notag\\
&+g_x(X^*(T))\Phi(T)\Psi(t)\left(b_u^*(t)-\sum_{j=1}^m\sigma_x^{j,*}(t)\sigma_u^{j,*}(t)\right)\notag\\
&+\sum_{j=1}^m\int_t^TD_t^j\left[f_x^*(s)\Phi(s)\right]ds\Psi(t)\sigma_u^{j,*}(t)\notag\\
&+\sum_{j=1}^mD_t^j\left[g_x(X^*(T))\Phi(T)\right]\Psi(t)\sigma_u^{j,*}(t)\Bigg]v(t)dt\notag\\
+\mathbb{E}\int_0^T&\sum_{j=1}^m\mathbb{D}_t^{H,j}\Bigg[\mathbb{E}^{\mathcal{F}_t}\left[\int_t^Tf_x^*(s)\Phi(s)ds\right]\Psi(t)\gamma_u^{j,*}(t)v(t)
\notag\\
&\qquad\quad\qquad+\mathbb{E}^{\mathcal{F}_t}\left[g_x(X^*(T))\Phi(T)\right]\Psi(t)\gamma_u^{j,*}(t)v(t)\Bigg] dt\ge 0
\end{align}
for $v(\cdot)=u(\cdot)-u^*(\cdot)$, where $u(\cdot)$ is any other admissible control process.
\end{theorem}
\begin{proof}
By Lemma \ref{lem1} and the dominated convergence theorem, we have 
\begin{align*}
\frac{d}{d\varepsilon}J\left(u^*(\cdot)+\varepsilon v(\cdot)\right)\Big|_{\varepsilon=0}=\mathbb{E}\int_0^T\left(f_x^*(t)y(t)+f_u^*(t)v(t)\right)dt+\mathbb{E}[g_x(X^*(T))y(T)].
\end{align*}
Substituting (\ref{3.9}) in the above equation we obtain
\begin{align} \label{3.12}
\frac{d}{d\varepsilon}J\Big(u^*(\cdot)&+\varepsilon v(\cdot)\Big)\Big|_{\varepsilon=0}\notag\\
=\mathbb{E}\Bigg[&\int_0^Tf_x^*(t)\Phi(t)\int_0^t\Psi(s)\left(b_u^*(s)-\sum_{j=1}^m\sigma_x^{j,*}(s)\sigma_u^{j,*}(s)\right)v(s)dsdt\notag\\
&+\sum_{j=1}^m\int_0^Tf_x^*(t)\Phi(t)\int_0^t\Psi(s)\sigma_u^{j,*}(s)v(s)dB_j(s)dt\notag\\
&+\sum_{j=1}^m\int_0^Tf_x^*(t)\Phi(t)\int_0^t\Psi(s)\gamma_u^{j,*}(s)v(s)\circ dB^H_j(s)dt+\int_0^Tf_u^*(t)v(t)dt\notag\\
&+g_x(X^*(T))\Phi(T)\int_0^T\Psi(t)\left(b_u^*(t)-\sum_{j=1}^m\sigma_x^{j,*}(t)\sigma_u^{j,*}(t)\right)v(t)dt\notag\\
&+\sum_{j=1}^mg_x(X^*(T))\Phi(T)\int_0^T\Psi(t)\sigma_u^{j,*}(t)v(t)dB_j(t)\notag\\
&+\sum_{j=1}^mg_x(X^*(T))\Phi(T)\int_0^T\Psi(t)\gamma_u^{j,*}(t)v(t)\circ dB_j^H(t)
\Bigg]\notag\\
=\mathbb{E}\Bigg[&\int_0^T\int_t^Tf_x^*(s)\Phi(s)ds\Psi(t)\left(b_u^*(t)-\sum_{j=1}^m\sigma_x^{j,*}(t)\sigma_u^{j,*}(t)\right)v(t)dt\notag\\
&+\sum_{j=1}^m\int_0^T\int_t^Tf_x^*(s)\Phi(s)ds\Psi(t)\sigma_u^{j,*}(t)v(t)dB_j(t)\notag\\
&+\sum_{j=1}^m\int_0^T\int_t^Tf_x^*(s)\Phi(s)ds\Psi(t)\gamma_u^{j,*}(t)v(t)\circ dB^H_j(t)+\int_0^Tf_u^*(t)v(t)dt\notag\\
&+g_x(X^*(T))\Phi(T)\int_0^T\Psi(t)\left(b_u^*(t)-\sum_{j=1}^m\sigma_x^{j,*}(t)\sigma_u^{j,*}(t)\right)v(t)dt\notag\\
&+\sum_{j=1}^mg_x(X^*(T))\Phi(T)\int_0^T\Psi(t)\sigma_u^{j,*}(t)v(t)dB_j(t)\notag\\
&+\sum_{j=1}^mg_x(X^*(T))\Phi(T)\int_0^T\Psi(t)\gamma_u^{j,*}(t)v(t)\circ dB_j^H(t)
\Bigg]\notag\\
=\mathbb{E}\Bigg[&\int_0^T G_1(T,t)v(t)dt+\sum_{j=1}^m\int_0^T G_{2,j}(T,t)v(t)dB_j(t)\notag\\
&\qquad+\sum_{j=1}^m\int_0^TG_{3,j}(T,t)v(t)\circ dB_j^H(t)\Bigg],
\end{align}
where
\begin{align*}
G_1(T,t)=&\int_t^Tf_x^*(s)\Phi(s)ds\Psi(t)\left(b_u^*(t)-\sum_{j=1}^m\sigma_x^{j,*}(t)\sigma_u^{j,*}(t)\right)
+f_u^*(t)\notag\\
&+g_x(X^*(T))\Phi(T)\Psi(t)\left(b_u^*(t)-\sum_{j=1}^m\sigma_x^{j,*}(t)\sigma_u^{j,*}(t)\right),
\end{align*}
\begin{align*}
G_{2,j}(T,t)=\int_t^Tf_x^*(s)\Phi(s)ds\Psi(t)\sigma_u^{j,*}(t)
+
g_x(X^*(T))\Phi(T)\Psi(t)\sigma_u^{j,*}(t)
\end{align*}
and
\begin{align*}
G_{3,j}(T,t)=\int_t^Tf_x^*(s)\Phi(s)ds\Psi(t)\gamma_u^{j,*}(t)
+
g_x(X^*(T))\Phi(T)\Psi(t)\gamma_u^{j,*}(t).
\end{align*}
Notice that $\int_t^Tf_x^*(s)\Phi(s)ds$ and $g_x(X^*(T))\Phi(T)$ are not $\mathcal{F}_t$-measurable, so the integral w.r.t $B^H(t)$ is hard to deal with in the following. To overcome this problem, we use martingale representation theorem and Fubini's theorem.

By martingale representation theorem, there is a unique $\mathcal{F}_t$-adapted process $M_s(\cdot)$, such that
\begin{align*}
f_x^*(s)\Phi(s)=\mathbb{E}^{\mathcal{F}_t}[f_x^*(s)\Phi(s)]+\int_t^sM_s(r)dB(r),\quad 0\le t\le s\le T.
\end{align*}
Then we have
\begin{align*}
&\mathbb{E}\int_0^T\int_t^Tf_x^*(s)\Phi(s)ds\Psi(t)\gamma_u^{j,*}(t)v(t)\circ dB^H_j(t)\\
&=\mathbb{E}\int_0^T\int_t^T\mathbb{E}^{\mathcal{F}_t}[f_x^*(s)\Phi(s)]ds\Psi(t)\gamma_u^{j,*}(t)v(t)\circ dB^H_j(t)\\
&\quad +\mathbb{E}\int_0^T\int_t^T\int_t^sM_s(r)dB(r)ds\Psi(t)\gamma_u^{j,*}(t)v(t)\circ dB^H_j(t)\\
&=\mathbb{E}\int_0^T\mathbb{E}^{\mathcal{F}_t}\left[\int_t^Tf_x^*(s)\Phi(s)ds\right]\Psi(t)\gamma_u^{j,*}(t)v(t)\circ dB^H_j(t)\\
&\quad +\mathbb{E}\int_0^T\int_r^T\int_0^rM_s(r)\Psi(t)\gamma_u^{j,*}(t)v(t)\circ dB^H_j(t)dsdB(r).
\end{align*}
The last term is $0$, because $\int_r^T\int_0^rM_s(r)\Psi(t)\gamma_u^{j,*}(t)v(t)\circ dB^H_j(t)ds$ is $\mathcal{F}_r$-measurable. Same method is used to $g_x(X^*(T))\Phi(T)$, we obtain
\begin{align*}
&\mathbb{E}\int_0^Tg_x(X^*(T))\Phi(T)\Psi(t)\gamma_u^{j,*}(t)v(t)\circ dB_j^H(t)\\
&=\mathbb{E}\int_0^T\mathbb{E}^{\mathcal{F}_t}\left[g_x(X^*(T))\Phi(T)\right]\Psi(t)\gamma_u^{j,*}(t)v(t)\circ dB_j^H(t),
\end{align*}
which means
\begin{align}\label{G3}
\mathbb{E}\int_0^TG_{3,j}(T,t)v(t)\circ dB_j^H(t)=\mathbb{E}\int_0^T\mathbb{E}^{\mathcal{F}_t}G_{3,j}(T,t)v(t)\circ dB_j^H(t)
\end{align}

~\\

Then we will deal with the last two terms of Eq.(\ref{3.12}). 
Firstly, we need the following identities from Malliavin calculus (Theorem 3.15 of \cite{nunno2008malliavin})
\begin{align*}
\mathbb{E}\left(F\int_0^Tg(t)dB_j(t)\right)=\mathbb{E}\int_0^T(D_t^jF)g(t)dt
,\end{align*}
and (Proposition 2.5 of \cite{han2013maximum})
\begin{equation*}
\mathbb{E}\left(F\int_a^b \psi_t  \circ dB^H_j(t)\right)=\mathbb{E}\int_a^b (\mathbb{D}_t^{H,j}[F \psi_t]) dt.
\end{equation*}
By these two identities, Fubini's theorem and Eq. (\ref{G3}), we have
\begin{align*}
\mathbb{E}
\int_0^T&G_{2,j}(T,t)v(t)dB_j(t)\notag\\
=&\mathbb{E}
\int_0^T\int_t^Tf_x^*(s)\Phi(s)ds\Psi(t)\sigma_u^{j,*}(t)v(t)dB_j(t)\notag\\
&+\mathbb{E}g_x(X^*(T))\Phi(T)\int_0^T\Psi(t)\sigma_u^{j,*}(t)v(t)dB_j(t)\notag\\
=&\mathbb{E}\int_0^T\int_0^tD_s^j\left[f_x^*(t)\Phi(t)\right]\Psi(s)\sigma_u^{j,*}(s)v(s)dsdt\notag\\
&+\mathbb{E}\int_0^TD_t^j\left[g_x(X^*(T))\Phi(T)\right]\Psi(t)\sigma_u^{j,*}(t)v(t)dt\notag\\
=&\mathbb{E}\int_0^T\int_t^TD_t^j\left[f_x^*(s)\Phi(s)\right]ds\Psi(t)\sigma_u^{j,*}(t)v(t)dt\notag\\
&+\mathbb{E}\int_0^TD_t^j\left[g_x(X^*(T))\Phi(T)\right]\Psi(t)\sigma_u^{j,*}(t)v(t)dt,
\end{align*}
and
\begin{align*}
\mathbb{E}\int_0^T&G_{3,j}(T,t)v(t)\circ dB_j^H(t)\notag\\
=&\mathbb{E}\int_0^T\mathbb{E}^{\mathcal{F}_t}\left[\int_t^Tf_x^*(s)\Phi(s)ds\right]\Psi(t)\gamma_u^{j,*}(t)v(t)\circ dB_j^H(t)\notag\\
&+\mathbb{E}\left[\mathbb{E}^{\mathcal{F}_t}\left[g_x(X^*(T))\Phi(T)\right]\int_0^T\Psi(t)\gamma_u^{j,*}(t)v(t)\circ dB_j^H(t)\right]\notag\\
=&\mathbb{E}\int_0^T\mathbb{D}_t^{H,j}\left[\mathbb{E}^{\mathcal{F}_t}\left[\int_t^Tf_x^*(s)\Phi(s)ds\right]\Psi(t)\gamma_u^{j,*}(t)v(t)\right] dt\notag\\
&+\mathbb{E}\int_0^T\mathbb{D}_t^{H,j}\left[\mathbb{E}^{\mathcal{F}_t}\left[g_x(X^*(T))\Phi(T)\right]\Psi(t)\gamma_u^{j,*}(t)v(t)\right] dt.\\
\end{align*}

Above all, we  rewrite (\ref{3.12}) as
\begin{align*}
\frac{d}{d\varepsilon}J\Big(u^*(\cdot)&+\varepsilon v(\cdot)\Big)\Big|_{\varepsilon=0}\\
=\mathbb{E}\int_0^T\Bigg[&\int_t^Tf_x^*(s)\Phi(s)ds\Psi(t)\left((b_u^*(t)-\sum_{j=1}^m\sigma_x^{j,*}(t)\sigma_u^{j,*}(t)\right)
+f_u^*(t)\notag\\
&+g_x(X^*(T))\Phi(T)\Psi(t)\left(b_u^*(t)-\sum_{j=1}^m\sigma_x^{j,*}(t)\sigma_u^{j,*}(t)\right)\notag\\
&+\sum_{j=1}^m\int_t^TD_t^j\left[f_x^*(s)\Phi(s)\right]ds\Psi(t)\sigma_u^{j,*}(t)\notag\\
&+\sum_{j=1}^mD_t^j\left[g_x(X^*(T))\Phi(T)\right]\Psi(t)\sigma_u^{j,*}(t)\Bigg]v(t)dt\notag\\
+\mathbb{E}\int_0^T&\sum_{j=1}^m\mathbb{D}_t^{H,j}\Bigg[\mathbb{E}^{\mathcal{F}_t}\left[\int_t^Tf_x^*(s)\Phi(s)ds\right]\Psi(t)\gamma_u^{j,*}(t)v(t)
\notag\\
&\qquad\quad\qquad+\mathbb{E}^{\mathcal{F}_t}\left[g_x(X^*(T))\Phi(T)\right]\Psi(t)\gamma_u^{j,*}(t)v(t)\Bigg] dt.
\end{align*}
Since $u^*(\cdot)$ is the optimal control, $\frac{d}{d\varepsilon}J\Big(u^*(\cdot)+\varepsilon v(\cdot)\Big)\Big|_{\varepsilon=0}\ge 0$. This complete the proof of Theorem \ref{mainin}.
\end{proof}

For simplicity, we  rewrite the $G\hat{a}teaux$ derivative of $J$ as
\begin{align}\label{G-J}
&\frac{d}{d\varepsilon}J\Big(u^*(\cdot)+\varepsilon v(\cdot)\Big)\Big|_{\varepsilon=0}\notag\\
&:=\mathbb{E}\int_0^T\Big[b_u^*(t)P_tv(t)+\sum_{j=1}^m\sigma_u^{j,*}(t)Q_j(t)v(t)+f_u^*(t)+\sum_{j=1}^m\mathbb{D}_t^{H,j}\left[\mathbb{E}^{\mathcal{F}_t}\left[\gamma_u^*(t)P(t)v(t)\right]\right]\Big]dt\notag\\
&=\mathbb{E}\int_0^T\Big[b_u^*(t)p_tv(t)+\sum_{j=1}^m\sigma_u^{j,*}(t)q_j(t)v(t)+f_u^*(t)+\sum_{j=1}^m\mathbb{D}_t^{H,j}\left[\gamma_u^*(t)p(t)v(t)\right]\Big]dt.
\end{align}
Here
\begin{align}\label{3.19}
P(t)=&\Psi(t)\left[\int_t^Tf_x^*(s)\Phi(s)ds+g_x(X^*(T))\Phi(T)\right],\notag\\
p(t)=&\Psi(t)\mathbb{E}^{\mathcal{F}_t}\left[\int_t^Tf_x^*(s)\Phi(s)ds+g_x(X^*(T))\Phi(T)\right],
\end{align}
and
\begin{align}\label{3.21}
Q_j(t)=&-\sigma_x^{j,*}(t)P(t)+\Psi(t)\left[\int_t^TD_t^j\big[f_x^*(s)\Phi(s)\big]ds+D_t^j\big[g_x(X^*(T))\Phi(T)\big]\right],\notag\\
q_j(t)=&-\sigma_x^{j,*}(t)p(t)+\Psi(t)\mathbb{E}^{\mathcal{F}_t}\left[\int_t^TD_t^j\big[f_x^*(s)\Phi(s)\big]ds+D_t^j\big[g_x(X^*(T))\Phi(T)\big]\right].
\end{align}
Since $\mathbb{D}^H_t[v(t)]=0$, by the arbitrary of $v(\cdot)$, we obtain the necessary condition for optimal control:
\begin{align*}
b_u^*(t)p_t+\sum_{j=1}^m\sigma_u^{j,*}(t)q_j(t)+f_u^*(t)+\sum_{j=1}^m\mathbb{D}_t^{H,j}\left[\gamma_u^*(t)p(t)\right]=0.
\end{align*}

\begin{remark}
Thanks for Eq. (\ref{G3}), we obtain the necessary condition only contains $p(\cdot)$, but not $P(\cdot)$. Indeed, we simplify the condition obtained in \cite{han2013maximum}, where the necessary condition contains the term $\mathbb{E}^{\mathcal{F}_t}\left[\mathbb{D}^{H,j}_t[\gamma_u^*(t)P(t)]\right]$.
\end{remark}

\begin{remark}
Up to now, we obtain the condition of optimal control when $v\in\bar{\mathbb{V}}$. Indeed, for any $\tilde{v}=\tilde{u}-u^*\in\mathbb{V}$ and $\delta>0$, there exists at least one $\bar{v}\in\bar{\mathbb{V}}$ such that $\mathbb{E}\|\tilde{v}(t)-\bar{v}(t)\|^2\le \delta$ for all $t\in[0,T]$. Then
\begin{align*}
&J(u^*(\cdot)+\tilde{v}(\cdot))-J(u^*(\cdot))\\&=J(u^*(\cdot)+\tilde{v}(\cdot))-J(u^*(\cdot)+\bar{v}(\cdot))+J(u^*(\cdot)+\bar{v}(\cdot))-J(u^*(\cdot))\\
&\ge O(\delta).
\end{align*}
By the arbitrary of $\tilde{v}$ and $\delta$, we show the $u^*(\cdot)$ is the optimal control in $\mathbb{U}$.
\end{remark}
~\\

Then we give out the relationship between $p$ and $q$.
~\\

\begin{theorem}\label{the2}
Let $p(\cdot)$, $q(\cdot)$ be defined as (\ref{3.19}) and (\ref{3.21}). Then $\left(p(\cdot), q(\cdot)\right)$ is the solution to the following BSDE,
\begin{align*}
\left\{\begin{array}{l}
-dp(t)=\left[b_x^*(t)p(t)+\sum_{j=1}^m\sigma_x^{j,*}(t)q_j(t)+f_x^*(t)\right]dt+\sum_{j=1}^m\gamma_x^{j,*}(t)p(t)\circ dB_j^H(t)\\
\qquad\qquad-\sum_{j=1}^m q_j(t)dB_j(t),\\
p(T)=g_x(X^*(T)).
\end{array}\right.
\end{align*}
\end{theorem}
\begin{proof}
It is obviously that
\begin{align*}
\mathbb{E}^{\mathcal{F}_t}\left[\int_0^Tf_x^*(s)\Phi(s)ds+g_x(X^*(T))\Phi(T)\right]
\end{align*}
is a martingale with respect to $\mathcal{F}_t$.
Due to martingale representation theorem, there are $\mathcal{F}_t$-adapted processes $\tilde{q}_j(\cdot)$, such that
\begin{align*}
p(t)=&\Psi(t)\mathbb{E}^{\mathcal{F}_t}\left[\int_0^Tf_x^*(s)\Phi(s)ds+g_x(X^*(T))\Phi(T)\right]-\Psi(t)\int_0^tf_x^*(s)\Phi(s)ds\\
=&-\sum_{j=1}^m\Psi(t)\int_0^t\Phi(s)\tilde{q}_j(s)dB_j(s)+\Psi(t)p(0)-\Psi(t)\int_0^tf_x^*(s)\Phi(s)ds\\
\triangleq&-\Psi(t)K(t),
\end{align*}
where
\begin{align}\label{3.24}
dK(t)=\Phi(t)f_x^*(t)dt+\sum_{j=1}^m\Phi(t)\tilde{q}_j(t)dB_j(t).
\end{align}
It should be noticed that, different from control problem only driven by fractional Brownian motion, $\tilde{q}_j(\cdot)$ and $q_j(\cdot)$ are different processes. We will find the relationship between them in the following proof. By (\ref{3.7}), (\ref{3.24}) and $\Psi(t)K(t)=-p(t)$, we get the following equation,
\begin{align}\label{3.25}
dp(t)=&-\Psi(t)dK(t)-K(t)d\Psi(t)-d\Psi(t)dK(t)\notag\\
=&-\left[b_x^*(t)p(t)-\sum_{j=1}^m\sigma_x^{j,*}(t)^2p(t)-\sum_{j=1}^m\sigma_x^{j,*}(t)\tilde{q}_j(t)+f_x^*(t)\right]dt\notag\\
&-\sum_{j=1}^m\gamma_x^{j,*}(t)p(t)\circ dB_j^H(t)-\sum_{j=1}^m\left(\sigma_x^{j,*}(t)p(t)+\tilde{q}_j(t)\right)dB_j(t),
\end{align}
which means
\begin{align*}
D_t^jp(t)=-\sigma_x^{j,*}(t)p(t)-\tilde{q}_j(t).
\end{align*}
On the other hand, by Lemma 3.3 of \cite{han2013maximum}, we see that
\begin{align*}
D_t^jp(t)=q_j(t).
\end{align*}
So we have 
\begin{align*}
\tilde{q}_j(t)=-\sigma_x^{j,*}(t)p(t)-q_j(t).
\end{align*}
Then Eq. (\ref{3.25}) can be rewritten as 
\begin{align*}
dp(t)=&-\left[b_x^*(t)p(t)+\sum_{j=1}^m\sigma_x^{j,*}(t)q_j(t)+f_x^*(t)\right]dt-\sum_{j=1}^m\gamma_x^{j,*}(t)p(t)\circ dB_j^H(t)\\
&+\sum_{j=1}^m q_j(t)dB_j(t).
\end{align*}
The terminal condition $p(T)=g_x(X^*(T))$ is directly obtained from (\ref{3.19}). This completes the proof of Theorem \ref{the2}.
\end{proof}

Then we give out the system of condition that the optimal control should satisfy.

\begin{theorem}
Let the assumptions (H1) and (H2) hold. Assume that $\left(u^*(\cdot),X^*(\cdot)\right)$ is an optimal pair. Then $\left(u^*(\cdot),X^*(\cdot)\right)$ satisfies the following system of equations.
\begin{align}
\left\{
\begin{array}{l}
 d X^*=b\left(t,X^*(t),u^*(t)\right)dt+\sigma\left(t,X^*(t),u^*(t)\right)dB(t)+\gamma\left(t,X^*(t),u^*(t)\right)\circ dB^H(t),\\\\
 X^*(0)=x_0,\\\\
-dp(t)=\left[b_x^*(t)p(t)+\sum_{j=1}^m\sigma_x^{j,*}(t)q_j(t)+f_x^*(t)\right]dt+\sum_{j=1}^m\gamma_x^{j,*}(t)p(t)\circ dB_j^H(t)\\\\
\qquad\qquad-\sum_{j=1}^m q_j(t)dB_j(t),\\\\
p(T)=g_x(X^*(T)),\\\\
b_u^*(t)p_t+\sum_{j=1}^m\sigma_u^{j,*}(t)q_j(t)+f_u^*(t)+\sum_{j=1}^m\mathbb{D}^{H,j}_t[\gamma^*_u(t)p(t)]=0.
\end{array}\right\}
\end{align}
\end{theorem}

\begin{remark}
Let $\sigma=0$, we obtain a new maximum principle for control systems driven by fractional Brownian motion. The main progress is that the necessary condition contains only one equality.
\end{remark}

~\\

\begin{remark}
Let $W(t)$ be a $1$-dimensional standard Brownain motion independent with the $1$-dimensional fractional Brownian motion $B^H(t)$. Consider the control system with the following state process,
\begin{align*}
    \left\{\begin{array}{l}
 d X(t)=b\left(t,X(t),u(t)\right)dt+\sigma\left(t,X(t),u(t)\right)dW(t)+\gamma\left(t,X(t),u(t)\right)\circ dB^H(t),\\
 X(0)=x_0.
\end{array}\right.
\end{align*}
Let $B(t)$ be the underlying Brownian motion of $B^H(t)$ and $W^H(t)=\int_0^t Z_H(t,\,s) dW(s)$. Let $\tilde{B}^H(t)=\left(B^H(t), W^H(t)\right)^T$ and  $\tilde{B}(t)=\left(B(t), W(t)\right)^T$. The the state process can be rewritten as
\begin{align*}
    \left\{\begin{array}{l}
 d X(t)=b\left(t,X(t),u(t)\right)dt+\tilde{\sigma}\left(t,X(t),u(t)\right)^Td\tilde{B}(t)+\tilde{\gamma}\left(t,X(t),u(t)\right)^T\circ d\tilde{B}^H(t),\\
 X(0)=x_0,
\end{array}\right.
\end{align*}
where
\begin{align*}
\tilde{\sigma}\left(t,X(t),u(t)\right)=\left(0,\sigma\left(t,X(t),u(t)\right)\right)^T,\quad \tilde{\gamma}\left(t,X(t),u(t)\right)=\left(\gamma\left(t,X(t),u(t)\right),0\right)^T.
\end{align*}
Then it could be dealt with through the main results we obtain.
\end{remark}

\section{Linear quadratic case}
In this section, we consider a linear quadratic (LQ) case, let the state process be defined as
\begin{align}\label{4.1}
    \left\{\begin{array}{l}
 d X(t)=\left(A(t)X(t)+\tilde{A}(t)u(t)\right)dt+\left(M(t)X(t)+\tilde{M}(t)u(t)\right)dB(t)\\
\qquad\qquad+\left(N(t)X(t)+\tilde{N}(t)u(t)\right)\circ dB^H(t),\quad t\in[0,T],\\
 X(0)=x_0,
\end{array}\right.
\end{align}
with the cost function
\begin{align}\label{4.2}
    J(u(\cdot))=\frac{1}{2}\mathbb{E}\left[\int_0^T \left(Q(t)X(t)^2+R(t)u(t)^2\right)dt+GX(T)^2\right].
\end{align}
Here $R(\cdot)$ is a positive determined function in $[0,T]$ and $G>0$.

By using the conclusion of section 3, we  get the adjoint equation 
\begin{align}\label{4.3}
\left\{\begin{array}{l}
-dp(t)=\left[A(t)p(t)+M(t)q(t)+Q(t)X^*(t)\right]dt+N(t)p(t)\circ dB^H(t)\\
\qquad\qquad-q(t)dB(t),\\
p(T)=GX^*(T),
\end{array}\right.
\end{align}
and the optimal control process $u^*(\cdot)$ should satisfy
\begin{align*}
\tilde{A}(t)p(t)+\tilde{M}(t)q(t)+R(t)u^*(t)+\tilde{N}(t)\mathbb{D}^{H}_t[p(t)]=0,
\end{align*}
i.e.,
\begin{align}\label{4.4}
u^*(t)=-R(t)^{-1}\left(\tilde{A}(t)p(t)+\tilde{M}(t)q(t)+\tilde{N}(t)\mathbb{D}^{H}_t[p(t)]\right).
\end{align}
\begin{theorem}
The function $u^*(t)=-R(t)^{-1}\left(\tilde{A}(t)p(t)+\tilde{M}(t)q(t)+\tilde{N}(t)\mathbb{D}^{H}_t[p(t)]\right),\quad t\in[0,T]$ is the unique optimal control for  LQ problem (\ref{4.1}), (\ref{4.2}), where $(p(\cdot),q(\cdot))$ is defined by Eq.  (\ref{4.3}).
\end{theorem}
\begin{proof}
We now prove $u^*(\cdot)$ is the optimal control. For any $\tilde{u}(\cdot)\in \mathbb{U}$, let $\tilde{X}(\cdot)$ and $ X^*(\cdot) $ be the corresponding state processes to $\tilde{u}(\cdot)$ and $u^*(\cdot)$, respectively. We have that 
\begin{align*}
&d\left[p(t)\left(\Tilde{X}(t)-X^*(t)\right)\right]\\=&p(t)d\left(\Tilde{X}(t)-X^*(t)\right)+\left(\Tilde{X}(t)-X^*(t)\right)dp(t)+dp(t)d\left(\Tilde{X}(t)-X^*(t)\right)\\
=&p(t)\left[A(t)\left(\Tilde{X}(t)-X^*(t)\right)+\tilde{A}(t)\left(\Tilde{u}(t)-u^*(t)\right)\right]dt\\
&+p(t)\left[N(t)\left(\Tilde{X}(t)-X^*(t)\right)+\tilde{N}(t)\left(\Tilde{u}(t)-u^*(t)\right)\right]\circ dB^H(t)\\
&-\left(\Tilde{X}(t)-X^*(t)\right)\left[A(t)p(t)+M(t)q(t)+Q(t)X^*(t)\right]dt\\
&-\left(\Tilde{X}(t)-X^*(t)\right)N(t)p(t)\circ dB^H(t)\\
&+q(t)\left[M(t)\left(\Tilde{X}(t)-X^*(t)\right)+\tilde{M}(t)\left(\Tilde{u}(t)-u^*(t)\right)\right]dt+F(t)dB(t)\\
=&-\left[R(t)u^*(t)(\tilde{u}(t)-u^*(t))+Q(t)X^*(t)\left(\Tilde{X}(t)-X^*(t)\right)\right]dt+F(t)dB(t)\\
&+p(t)\tilde{N}(t)\left(\Tilde{u}(t)-u^*(t)\right)\circ dB^H(t)-\mathbb{D}^H_t[p(t)]\tilde{N}(t)\left(\Tilde{u}(t)-u^*(t)\right)dt
\end{align*}
where $F(t)$ is a $\mathcal{F}_t$-adapted process and the last equality holds is because of Eq. (\ref{4.4}).
Taking the integral and expectation to the above equation, we get
\begin{align*}
&\mathbb{E}GX^*(T)\left(\Tilde{X}(T)-X^*(T)\right)\\
=&\mathbb{E}p(T)\left(\Tilde{X}(T)-X^*(T)\right)\\
=&\mathbb{E}\int_0^Td\left[p(t)\left(\Tilde{X}(t)-X^*(t)\right)\right]\\
=&-\mathbb{E}\int_0^T\left[R(t)u^*(t)(\tilde{u}(t)-u^*(t))+Q(t)\left(\Tilde{X}(t)-X^*(t)\right)\right]dt.
\end{align*}
Then using the fact $a^2-b^2\ge 2b(a-b)$ and by the above equation, we obtain
\begin{align*}
J\left(\tilde{u}(\cdot)\right)-J\left(u^*(\cdot)\right)=&\frac{1}{2}\mathbb{E}\int_0^T\left[Q(t)\left(\Tilde{X}(t)^2-X^*(t)^2\right)+R(t)\left(\Tilde{u}(t)^2-u^*(t)^2\right)\right]dt\\
&+\frac{1}{2}\mathbb{E}G\left(\Tilde{X}(T)^2-X^*(T)^2\right)\\
\ge&\mathbb{E}\int_0^T\left[Q(t)X^*(t)\left(\Tilde{X}(t)-X^*(t)\right)+R(t)u^*(t)\left(\Tilde{u}(t)-u^*(t)\right)\right]dt\\
&+\mathbb{E}GX^*(T)\left(\Tilde{X}(T)-X^*(T)\right)\\
=&0,
\end{align*}
which shows $u^*(t)$ is the optimal control.

Then we prove $u^*(t)$ is unique. Assume that both $u^{*,1}(t)$ and $u^{*,2}(t)$ are optimal controls, $X^1(t)$ and $X^2(t)$ are corresponding state processes, respectively. It is easy to check $\frac{X^1(t)+X^2(t)}{2}$ is the corresponding state process to $\frac{u^{*,1}(t)+u^{*,2}(t)}{2}$. We assume there exist constants $\delta>0, \alpha\ge 0$, such that $R(t)\ge \delta$ and 
\begin{align*}
    J(u^{*,1}(\cdot))=J(u^{*,2}(\cdot))=\alpha.
\end{align*}
Through the fact $a^2+b^2=2[(\frac{a+b}{2})^2+(\frac{a-b}{2})^2]$, we have that
\begin{align*}
2\alpha=&J(u^{*,1}(\cdot))+J(u^{*,2}(\cdot))\\
=&\frac{1}{2}\mathbb{E}\int_0^T \Big[Q(t)(X^1(t)X^1(t)+X^2(t)X^2(t))+R(t)(u^{*,1}(t)u^{*,1}(t)+u^{*,2}(t)u^{*,2}(t))\Big]dt\\
&+\frac{1}{2}\mathbb{E}G(X^1(T)X^1(T)+X^2(T)X^2(T))\\
\ge&\mathbb{E}\int_0^T \Big[Q(t)\Big(\frac{X^1(t)+X^2(t)}{2}\Big)^2+R(t)\Big(\frac{u^{*,1}(t)+u^{*,2}(t)}{2}\Big)^2\Big]dt\\
&+\mathbb{E}G\Big(\frac{X^1(T)+X^2(T)}{2}\Big)^2+\mathbb{E}\int_0^TR(t)\Big(\frac{u^{*,1}(t)-u^{*,2}(t)}{2}\Big)^2dt\\
=&2J\Big(\frac{u^{*,1}(t)+u^{*,2}(t)}{2}\Big)+\mathbb{E}\int_0^TR(t)\Big(\frac{u^{*,1}(t)-u^{*,2}(t)}{2}\Big)^2dt\\
\ge&2\alpha+\frac{\delta}{4}\mathbb{E}\int_0^T|u^{*,1}(t)-u^{*,2}(t)|^2dt.
\end{align*}
Thus, we have 
\begin{align*}
    \mathbb{E}\int_0^T|u^{*,1}(t)-u^{*,2}(t)|^2dt\le 0,
\end{align*}
which shows that the optimal control is unique.
\end{proof}

\section{Funding and conflict of interest}
This work was supported by National Key R$\&$D Program of China (Grant numbers 2023YFA1009200) and National Science Foundation of China (Grant numbers 12471417). The author Yuecai Han has received research support from the fundings. There is no conflict of interest.
\bibliography{main}

\end{document}